\documentclass[12pt]{amsart}
\newtheorem{theorem}{Theorem}[section]
\newtheorem{lemma}[theorem]{Lemma}

\newtheorem{proposition}[theorem]{Proposition}

\theoremstyle{definition}
\newtheorem{definition}[theorem]{Definition}
\newtheorem{remark}[theorem]{Remark}

\numberwithin{equation}{section}

\begin{document}
\baselineskip=15.5pt

\title[Tangent bundle of wonderful compactification of adjoint group]{Stability
of the tangent bundle of the wonderful compactification of an adjoint group}

\author[I. Biswas]{Indranil Biswas}

\address{School of Mathematics, Tata Institute of Fundamental
Research, Homi Bhabha Road, Bombay 400005, India}

\email{indranil@math.tifr.res.in}

\author[S. S. Kannan]{S. Senthamarai Kannan}

\address{Chennai Mathematical Institute, H1, SIPCOT IT Park, Siruseri,
Kelambakkam 603103, India}

\email{kannan@cmi.ac.in}



\date{}

\begin{abstract}
Let $G$ be a complex linear algebraic group which is simple of adjoint type.
Let $\overline G$ be the wonderful compactification of $G$. We prove that
the tangent bundle of $\overline G$ is stable with respect to every polarization
on $\overline G$.
\end{abstract}

\maketitle

\section{Introduction}\label{sec1}

De Concini and Procesi constructed compactifications of complex simple groups of adjoint
type, which are known as wonderful compactifications. These
compactifications have turned out to be very
useful objects. Our aim here is to investigate equivariant vector bundles on a
wonderful compactification. One of the key concepts associated to a vector bundle on
a projective variety is the notion of stability introduced by Mumford.

We prove the following (see Theorem \ref{thm1}):

\begin{theorem}\label{thm0}
Let $\overline G$ be the wonderful compactification of a complex simple group $G$
of adjoint type. Take any polarization $L$ on $\overline G$. Then the
tangent bundle of $\overline G$ is stable with respect to $L$.
\end{theorem}

Theorem \ref{thm0} is proved using a result proved here on equivariant vector
bundles over $\overline G$ which we will now explain.

Take $G$ as in Theorem \ref{thm0}.
Let $\widetilde G$ be the universal cover of $G$. The action of $G\times G$ on
$\overline G$ produces an action of ${\widetilde G}\times {\widetilde G}$ on
$\overline G$. A holomorphic vector bundle on $\overline G$ is called equivariant
if it is equipped with a lift of the action of ${\widetilde G}\times {\widetilde G}$;
see Definition \ref{def1} for the details. Let $e_0\,\in\, G$ be the identity element.
The group $\widetilde G$ is the connected component, containing the identity element,
of the isotropy group of $e_0$ for the action of ${\widetilde G}\times {\widetilde G}$
on $\overline G$. If $(E\, ,\gamma)$ is an equivariant vector bundle on $\overline G$,
then the action $\gamma$ of ${\widetilde G}\times {\widetilde G}$ on $E$ produces an
action of $\widetilde G$ on the fiber $E_{e_0}$.

We prove the following (see Proposition \ref{prop2}):

\begin{proposition}\label{prop0}
Let $(E\, ,\gamma)$ be an equivariant vector bundle of rank $r$ on $\overline G$
such that the action of $\widetilde{G}$ on $E_{e_0}$ is irreducible. Then either
$E$ is stable or there is a holomorphic line bundle $\xi$ on $\overline G$ such that
$E$ is holomorphically isomorphic to $\xi^{\oplus r}$.
\end{proposition}

We show that the tangent bundle $T\overline G$ is not isomorphic to
$\xi^{\oplus d}$, where $\xi$ is some holomorphic line bundle on $\overline G$,
and $d\,=\, \dim_{\mathbb C} G$. In view of this result, Theorem \ref{thm0}
follows from Proposition \ref{prop0}.

A stable vector bundle admits an irreducible Einstein--Hermitian connection. It
would be very interesting to be able to describe the Einstein--Hermitian
structure of the tangent bundle of $\overline G$.

In \cite{Ka}, Kato has carried out a detailed investigation of the equivariant
bundles on partial compactifications of reductive groups.

\section{Equivariant vector bundles on $\overline G$}

Let $G$ be a connected linear algebraic group defined over $\mathbb C$ such that
the Lie algebra of $G$ is simple and the center of $G$ is trivial. In other
words, $G$ is simple of adjoint type. The group $G\times G$ acts on $G$: the
action of any $(g_1\, ,g_2)\,\in\, G\times G$ is
the map $y\, \longmapsto\, g_1yg^{-1}_2$.
Let $\overline{G}$ be the wonderful compactification of $G$ \cite{DP}.
A key property of the wonderful compactification of $G$ is that the above action
of $G\times G$ on $G$ extends to an action of $G\times G$ on $\overline{G}$. Let
\begin{equation}\label{e0}
\pi\, :\, \widetilde{G}\, \longrightarrow\, G
\end{equation}
be the universal cover. Using the projection $\pi$ in \eqref{e0}, the above
mentioned action of $G\times G$ on $\overline{G}$ produces an action of
$\widetilde{G}\times \widetilde{G}$ on $\overline{G}$
\begin{equation}\label{e1}
\beta\, :\,\widetilde{G}\times \widetilde{G}\,\longrightarrow\,
\text{Aut}^0(\overline{G})\, ,
\end{equation}
where $\text{Aut}^0(\overline{G})$ is the connected component, containing the
identity element, of the group of automorphisms of the variety $\overline{G}$.

\begin{definition}\label{def1}
An {\em equivariant} vector bundle on $\overline G$ is a pair $(E\, ,\gamma)$,
where $E$ is a holomorphic vector bundle on $\overline G$ and
$$
\gamma\, :\, \widetilde{G}\times \widetilde{G}\times E\,\longrightarrow\, E
$$
is a holomorphic action of $\widetilde{G}\times \widetilde{G}$ on the total
space of $E$, such that the following two conditions hold:
\begin{enumerate}
\item the projection of $E$ to $\overline G$ intertwines the actions of
$\widetilde{G}\times \widetilde{G}$ on $E$ and $\overline G$, and

\item the action of $\widetilde{G}\times \widetilde{G}$ on $E$ preserves the
linear structure of the fibers of $E$.
\end{enumerate}
\end{definition}

Note that the first condition in Definition \ref{def1} implies that the action of any
$g\,\in\, \widetilde{G}\times \widetilde{G}$ sends a fiber $E_x$ to the fiber
$E_{\beta(g)(x)}$, where $\beta$ is the homomorphism in \eqref{e1}. The second
condition in Definition \ref{def1} implies that the self-map of $E$ defined by $v\,
\longmapsto\,\gamma(g\, ,v)$ is a holomorphic isomorphism of the vector bundle $E$
with the pullback $\beta(g^{-1})^*E$. Therefore, if $(E\, ,\gamma)$ is an equivariant
vector bundle on $\overline G$, then for every $g\,\in\,
\widetilde{G}\times \widetilde{G}$, the pulled back holomorphic vector bundle
$\beta(g)^*E$ is holomorphically isomorphic to $E$. The following proposition
is a converse statement of it.

\begin{proposition}\label{prop1}
Let $E$ be a holomorphic vector bundle on $\overline G$ such that
for every $g\,\in\,
\widetilde{G}\times \widetilde{G}$, the pulled back holomorphic vector bundle
$\beta(g)^*E$ is holomorphically isomorphic to $E$. Then there is a holomorphic
action $\gamma$ of $\widetilde{G}\times \widetilde{G}$ on $E$ such that the pair
$(E\, ,\gamma)$ is an equivariant vector bundle on $\overline G$.
\end{proposition}

\begin{proof}
Let $\text{Aut}(E)$ denote the group of holomorphic automorphisms of the vector bundle
$E$ over the
identity map of $\overline G$. This set $\text{Aut}(E)$ is the Zariski open subset
of the affine space $H^0({\overline G},\, E\otimes E^\vee)$ defined by the
locus of invertible
endomorphisms of $E$. Therefore, $\text{Aut}(E)$ is a connected complex algebraic group.

Let $\widetilde{\text{Aut}}(E)$ denote the set of all pairs of the form
$(g\, ,f)$, where $g\, \in\, \widetilde{G}\times \widetilde{G}$ and
$$
f\, :\, \beta(g^{-1})^*E\,\longrightarrow\, E
$$
is a holomorphic isomorphism of vector bundles. This set $\widetilde{\text{Aut}}(E)$
has a tautological structure of a group
$$
(g_2\, ,f_2)\cdot (g_1\, ,f_1)\,=\, (g_2g_1\, , f_2\circ f_1)\, .
$$
We will show that it is a connected complex algebraic group.

Let $p_1\, :\, \widetilde{G}\times \widetilde{G}\times{\overline G}\,\longrightarrow
\, {\overline G}$ be the projection to the last factor. Let
$$
\widehat{\beta}\, :\, \widetilde{G}\times \widetilde{G}\times{\overline G}\,
\longrightarrow\, {\overline G}
$$
be the algebraic morphism defined by $(g\, ,y)\,\longmapsto\, \beta(g^{-1})(y)$,
where $g\,\in\,\widetilde{G}\times \widetilde{G}$ and $y\, \in\,
{\overline G}$. Let
$$
q\, :\, \widetilde{G}\times \widetilde{G}\times{\overline G}\,\longrightarrow\,
\widetilde{G}\times \widetilde{G}
$$
be the projection to the first two factors. Now consider the direct image
$$
{\mathcal E}\, :=\, q_*((p^*_1E)\otimes (\widehat{\beta}^*E)^\vee)\,\longrightarrow\,
\widetilde{G}\times \widetilde{G}\, .
$$
It is locally free. The set $\widetilde{\text{Aut}}(E)$ is a Zariski open subset
of the total space of the algebraic vector bundle $\mathcal E$. Therefore,
$\widetilde{\text{Aut}}(E)$ is a connected complex algebraic group.

The Lie algebra of $G$ will be denoted by $\mathfrak g$. The Lie algebra of
$\widetilde{\text{Aut}}(E)$ will be denoted by $A(E)$.
We have a short exact sequence of groups
\begin{equation}\label{e2}
e\,\longrightarrow\,\text{Aut}(E)\,\longrightarrow\,\widetilde{\text{Aut}}(E) \,
\stackrel{\rho}{\longrightarrow}\,\widetilde{G}\times \widetilde{G}\,
\longrightarrow\,e\, ,
\end{equation}
where $\rho$ sends any $(g\, ,f)$ to $g$. Let
\begin{equation}\label{e3}
\rho'\, :\, A(E)\,\longrightarrow\,{\mathfrak g}\oplus {\mathfrak g}
\end{equation}
be the homomorphism of Lie algebras corresponding to
$\rho$ in \eqref{e2}. Since ${\mathfrak g}
\oplus {\mathfrak g}$ is semisimple, there is a homomorphism of Lie algebras
$$
\tau\, :\, {\mathfrak g}\oplus {\mathfrak g}\,\longrightarrow\,A(E)
$$
such that
\begin{equation}\label{es}
\rho'\circ\tau\,=\, \text{Id}_{{\mathfrak g}\oplus {\mathfrak g}}
\end{equation}
\cite[p. 91, Corollaire 3]{Bo}. Fix such a homomorphism $\tau$ satisfying \eqref{es}.
Since the group $\widetilde{G}\times \widetilde{G}$
is simply connected, there is a unique holomorphic homomorphism
$$
\widetilde{\tau}\, :\, \widetilde{G}\times \widetilde{G}\,\longrightarrow\,
\widetilde{\text{Aut}}(E)
$$
such that the corresponding homomorphism of Lie algebras coincides with
$\tau$. From \eqref{es} it follows immediately that $\rho\circ\widetilde{\tau}\,=\,
{\rm Id}_{\widetilde{G}\times\widetilde{G}}$.

We now note that $\widetilde{\tau}$ defines an action of $\widetilde{G}\times
\widetilde{G}$ on $E$. The pair $(E\, ,\widetilde{\tau})$ is an equivariant
vector bundle.
\end{proof}

\section{Irreducible representations and stability}

Fix a very ample class $L\, \in\, \text{NS}(\overline{G})$, where $\text{NS}(
\overline{G})$ is the N\'eron--Severi group of $\overline{G}$. The degree of
a torsionfree coherent sheaf $F$ on $\overline{G}$ is defined to be
$$
\text{degree}(F)\, :=\, (c_1(F)\cup c_1(L)^{d-1})\cap [\overline{G}]
\,\in\,\mathbb Z\, ,
$$
where $d\,=\, \dim_{\mathbb C}G$. If $\text{rank}(F)\, \geq\, 1$, then
$$
\mu(F)\,:=\, \frac{\text{degree}(F)}{\text{rank}(F)}\,\in\, \mathbb Q
$$
is called the \textit{slope} of $F$.

A holomorphic vector bundle $F$ over $\overline{G}$ is called \textit{stable}
(respectively, \textit{semistable}) if for every nonzero coherent subsheaf $F'\,
\subset\, F$ with $\text{rank}(F')\, <\,\text{rank}(F)$, the inequality
$$
\mu(F')\, <\, \mu(F) ~\ \text{ (respectively, }~\ \mu(F')\, \leq\, \mu(F){\rm )}
$$
holds. A holomorphic vector bundle on $\overline{G}$ is called \textit{polystable} if
it is a direct sum of stable vector bundles of same slope.

Let
$$
e_0\,\in\, G\, \subset\, \overline{G}
$$
be the identity element. Let $\text{Iso}_{e_0}\,\subset\, \widetilde{G}\times
\widetilde{G}$ be the isotropy subgroup of $e_0$ for the action of $\widetilde{G}
\times \widetilde{G}$ on $\overline{G}$. The connected component of $\text{Iso}_{e_0}$
containing the identity element is $\widetilde{G}$.

If $(E\, ,\gamma)$ is an equivariant vector bundle on $\overline G$, then $\gamma$
gives an action of $\text{Iso}_{e_0}$ on the fiber $E_{e_0}$. In particular,
we get an action of $\widetilde{G}$ on $E_{e_0}$.

\begin{lemma}\label{lem1}
Let $(E\, ,\gamma)$ be an equivariant vector bundle on $\overline G$ such that the
above action of $\widetilde{G}$ on $E_{e_0}$ is irreducible. Then
the vector bundle $E$ is polystable.
\end{lemma}

\begin{proof}
Let
\begin{equation}\label{e4}
E_1 \, \subset\,\cdots \, \subset\, E_n\,=\, E
\end{equation}
be the Harder--Narasimhan filtration of $E$ \cite[p. 16, Theorem 1.3.4]{HL}.
Since $\widetilde{G}\times \widetilde{G}$ is connected, the action of
$\widetilde{G}\times \widetilde{G}$ on $\overline G$ preserves the
N\'eron--Severi class $L$. Therefore, the filtration in \eqref{e4}
is preserved by the action of $\widetilde{G}\times \widetilde{G}$ on $E$. Note that
$(E_1)_{e_0}\,\not=\,0$ because in that case $E_1\vert_G\,=\, 0$ by the
equivariance of $E$, which in turn implies that $E_1\,=\, 0$. Now,
from the irreducibility of the action of $\widetilde{G}$ on $E_{e_0}$ we conclude
that $(E_1)_{e_0}\,=\, E_{e_0}$. In particular, $\text{rank}(E_1)\,=\,
\text{rank}(E)$. This implies that $E_1\,=\, E$. Hence $E$ is semistable.

Let
$$
F\, \subset\, E
$$
be the unique maximal polystable subsheaf of the semistable vector bundle $E$
\cite[p. 23, Theorem 1.5.9]{HL}. From the uniqueness of $F$ and the connectivity
of $\widetilde{G}\times \widetilde{G}$ we conclude that $F$ is preserved by the action
of $\widetilde{G}\times \widetilde{G}$ on $E$. Just as done above, using the
irreducibility of the action of $\widetilde{G}$ on $E_{e_0}$ we conclude
that $F_{e_0}\,=\, E_{e_0}$. Hence $F\,=\, E$, implying that $E$ is polystable.
\end{proof}

The following lemma is well-known.

\begin{lemma}\label{lem2}
Let $V_1$ and $V_2$ be two finite dimensional irreducible complex
$\widetilde G$--modules such that both $V_1$ and $V_2$ are nontrivial.
Then the $\widetilde G$--module $V_1\otimes V_2$ is not irreducible.
\end{lemma}

Lemma \ref{lem2} is a very special case of the PRV conjecture, \cite{PRV},
which is now proved. We also note that Lemma \ref{lem2} is an
immediate consequence of \cite[p. 683, Theorem 1]{Ra}.

\begin{proposition}\label{prop2}
Let $(E\, ,\gamma)$ be an equivariant vector bundle of rank $r$ on $\overline G$
such that the action of $\widetilde{G}$ on $E_{e_0}$ is irreducible. Then either
$E$ is stable or there is a holomorphic line bundle $\xi$ on $\overline G$ such that
$E$ is holomorphically isomorphic to $\xi^{\oplus r}$.
\end{proposition}

\begin{proof}
{}From Lemma \ref{lem1} we know that $E$ is polystable. Therefore, there are
distinct stable vector bundles $F_1\, ,\cdots\, , F_\ell$ and positive integers
$n_1\, ,\cdots\, , n_\ell$, such that $\mu(F_i)\,=\, \mu(E)$ for every $i$ and
\begin{equation}\label{e5}
E\,=\, \bigoplus_{i=1}^\ell F^{\oplus n_i}_i\, .
\end{equation}
We emphasize that $F_i\,\not=\, F_j$ if $i\,\not=\, j$. The vector bundles
$F_1\, ,\cdots\, ,F_\ell$ are uniquely determined by $E$ up to a permutation of
$\{1\, ,\cdots\, ,\ell\}$ \cite[p. 315, Theorem 2]{At}.

Fix a holomorphic isomorphism between the two vector bundles in the two sides of
\eqref{e5}. Take $F_i$ and $F_j$ with $i\, \not=\, j$. Since they
are nonisomorphic stable vector bundles of same slope, we have
$$
H^0(\overline{G},\, F_j\otimes F^\vee_i)\,=\, 0\,=\,
H^0(\overline{G},\, F_i\otimes F^\vee_j)\, .
$$
Consequently, for every $i\,\in\, \{1\, ,\cdots\, ,\ell\}$, there is a unique
subbundle of $E$ which is isomorphic to $F^{\oplus n_i}_i$. Using this it follows
that for any $g\, \in\, \widetilde{G}\times \widetilde{G}$, and any
$j\,\in\, \{1\, ,\cdots\, ,\ell\}$, there is a $k\,\in\, \{1\, ,\cdots\, ,\ell\}$
such that the action of $g$ on $E$ takes the subbundle $F^{\oplus n_j}_j$
to $F^{\oplus n_k}_k$. Since $\widetilde{G}
\times\widetilde{G}$ is connected, this implies that the action $\gamma$ of
$\widetilde{G}\times \widetilde{G}$ on $E$ preserves the subbundle $F^{\oplus
n_i}_i$ for every $i\,\in\, \{1\, ,\cdots\, ,\ell\}$. Now from the irreducibility
of the action of $\widetilde G$ on $E_{e_0}$ we conclude that $\ell\,=\,1$.

We will denote $F_1$ and $n_1$ by $F$ and $n$ respectively. So
\begin{equation}\label{e6}
F\,=\, F^{\oplus n}\, .
\end{equation}

Since for every $g\,\in\, \widetilde{G}\times \widetilde{G}$, the pulled back
holomorphic vector bundle $\beta(g)^*(F^{\oplus n})$ is holomorphically isomorphic
to $F^{\oplus n}$, using \cite[p. 315, Theorem 2]{At} and the fact that
$F$ is indecomposable (recall that $F$ is stable), we conclude that
the pulled back holomorphic vector bundle $\beta(g)^*F$ is holomorphically isomorphic
to $F$ for every $g\,\in\, \widetilde{G}\times \widetilde{G}$. Therefore, from
Proposition \ref{prop1} we know that there is an action $\delta$ of $\widetilde{G}
\times\widetilde{G}$ on $F$ such that $(F\, ,\delta)$ is an equivariant vector bundle
on $\overline G$.

The actions $\gamma$ and $\delta$ together define an action of $\widetilde{G}\times
\widetilde{G}$ on the vector bundle $$Hom(F\, ,E)\,=\, E\otimes F^\vee\, .$$ This
action of $\widetilde{G}\times\widetilde{G}$ on $Hom(F\, ,E)$ produces an
action of $\widetilde{G}\times\widetilde{G}$ on the vector space
$H^0(\overline{G},\, Hom(F\, ,E))$.

In view of \eqref{e6}, we have a canonical isomorphism
\begin{equation}\label{e7}
E\,=\, F\otimes_{\mathbb C} H^0(\overline{G},\, Hom(F\, ,E))\, .
\end{equation}
This isomorphism sends any $(v\, , \sigma)\, \in\, (F_x\, ,
H^0(\overline{G},\, Hom(F\, ,E)))$ to the evaluation $\sigma_x(v)\,\in\, E_x$.
The isomorphism in \eqref{e7} is $\widetilde{G}\times \widetilde{G}$--equivariant.
Since the action of $\widetilde{G}$ on $E_{e_0}$ is irreducible, from
Lemma \ref{lem2} we conclude that either $\text{rank}(F)\,=\, 1$ or
$$
\dim H^0(\overline{G},\, Hom(F\, ,E))\,=\, 1\, .
$$

If $\dim H^0(\overline{G},\, Hom(F\, ,E))\,=\, 1$, then from \eqref{e7}
and the fact that $F$ is stable it follows immediately that $E$ is stable.
If $\text{rank}(F)\,=\, 1$, then from \eqref{e7} it follows that
$$
E\,=\, F^{\oplus r}\, ,
$$
where $r\,=\, \text{rank}(E)$.
\end{proof}

\begin{remark}\label{rem1}
Let $V$ be any irreducible $G$--module. Consider the trivial right action of
$\widetilde G$ on $V$ as well as the left action of $\widetilde G$ on $V$ given by the
combination of the action of $G$ on $V$ and the projection in \eqref{e0}. Therefore,
we get the diagonal action of ${\widetilde G}\times\widetilde G$ on the trivial
vector bundle $\overline{G}\times V$ over $\overline{G}$. Consequently, the trivial
vector bundle $\overline{G}\times V$ gets the structure on an equivariant vector bundle.
We note that the action of $\widetilde{G}\, \subset\, \text{Iso}_{e_0}$
on the fiber of $\overline{G}\times V$ over the point $e_0$ is irreducible because the
$G$--module $V$ is irreducible.
\end{remark}

\section{The tangent bundle}

\begin{theorem}\label{thm1}
Let $L\, \in\, {\rm NS}(\overline{G})$ be any ample class.
The tangent bundle of $\overline G$ is stable with respect to $L$.
\end{theorem}

\begin{proof}
Since $G$ is simple, the adjoint action of $\widetilde G$ on the Lie algebra
$\mathfrak g$ of $G$ is irreducible. In view of Proposition \ref{prop2}, it suffices
to show that the tangent bundle $T\overline{G}$ is not of the form
$\xi^{\oplus d}$, where $\xi$ is a holomorphic line bundle on $\overline{G}$
and $d\,=\, \dim_{\mathbb C} G$.

Assume that
\begin{equation}\label{e8}
T\overline{G}\,=\, \xi^{\oplus d}\, ,
\end{equation}
where $\xi$ is a holomorphic line bundle on $\overline{G}$.

Since the variety $G$ is unirational (cf. \cite[p. 4, Theorem 3.1]{Gi}), the
compactification $\overline{G}$ is also unirational. Hence $\overline{G}$ is simply
connected \cite[p. 483, Proposition 1]{Se}. As $T\overline{G}$ holomorphically splits
into a direct sum of line bundles (see \eqref{e8}) and $\overline{G}$ is simply
connected, it follows that
$$
\overline{G}\,=\, ({\mathbb C}{\mathbb P}^1)^d
$$
\cite[p. 242, Theorem 1.2]{BPT}. But the tangent bundle of $({\mathbb C}{\mathbb
P}^1)^d$ is not of the form $\xi^{\oplus d}$ (see \eqref{e8}); although the tangent
bundle of $({\mathbb C}{\mathbb P}^1)^d$ is a direct sum of line bundles, the line
bundles in its decomposition are not isomorphic. Therefore, $T\overline{G}$ is not
of the form $\xi^{\oplus d}$. This completes the proof.
\end{proof}

\section*{Acknowledgements}
We thank the referee and D. S. Nagaraj for some comments.
This work was carried out when both the authors were visiting KSOM, Calicut. We
thank KSOM for hospitality. The first-named author acknowledges the support of the
J. C. Bose Fellowship.


\end{document}